\documentclass[reqno,11pt]{amsart}
\numberwithin{equation}{section}
\usepackage{amsmath}
\usepackage{esint} 
\usepackage{amsthm}
\usepackage{epsfig}
\usepackage{psfrag}
\usepackage{graphicx}
\graphicspath{{figure/}}
\usepackage{graphpap,latexsym,epsf}
\usepackage{color}
\usepackage{amssymb,mathrsfs,enumerate}
\usepackage{mathtools}
\usepackage{subfigure}
\usepackage{etoolbox}
\cslet{blx@noerroretextools}\empty
\usepackage[giveninits=true,style=ext-alphabetic,backend=biber]{biblatex}
\definecolor{citegreen}{rgb}{0,0.6,0}
\definecolor{refred}{rgb}{0.8,0,0}
\usepackage[colorlinks, citecolor=citegreen, linkcolor=refred]{hyperref}
\usepackage{xurl}
\hypersetup{breaklinks=true}
\usepackage{cleveref}
\crefname{definition}{Definition}{Definitions}
\crefname{theorem}{Theorem}{Theorems}
\crefname{thmx}{Theorem}{Theorems}
\crefname{lemma}{Lemma}{Lemmas}
\crefname{step}{Step}{Steps}
\crefname{substep}{Step}{Steps}
\crefname{claim}{Claim}{Claims}
\crefname{proposition}{Proposition}{Propositions}
\crefname{corollary}{Corollary}{Corollaries}
\crefname{remark}{Remark}{Remarks}
\crefname{section}{Section}{Sections}
\crefname{subsection}{Section}{Sections}
\crefname{chapter}{Chapter}{Chapters}
\crefname{appendix}{Appendix}{Appendices}
\crefname{equation}{}{}

\usepackage{autonum}
\newcommand{\R}{\mathbb{R}}

\newcommand{\N}{\mathbb{N}}

\def\HHH{{\rm H}}
\def\RRR{{\mathrm R}}

\newcommand{\pa}{\partial}

\newcommand{\ep}{\varepsilon}

\newcommand{\Ric}{{\rm Ric}}

\newcommand{\na}{\nabla}

\newcommand{\hhh}{{\rm h}}

\renewcommand{\epsilon}{\varepsilon}
\renewcommand{\phi}{\varphi}

\newcommand{\Sf}{\mathbb{S}}

\mathchardef\emptyset="001F

\DeclarePairedDelimiter{\abs}{\lvert}{\rvert}

\definecolor{vgreen}{rgb}{0.1,0.5,0.2}
\definecolor{viola}{RGB}{85,26,139}
\renewcommand{\theequation}{\thesection.\arabic{equation}}

\newtheorem{theorem}{Theorem}[section]
\newtheorem{remark}[theorem]{Remark}
\newtheorem{corollary}[theorem]{Corollary}

\newtheorem{proposition}[theorem]{Proposition}

\newtheorem{lemma}[theorem]{Lemma}

\title[Equality in Heintze-Karcher]{The equality case in the substatic Heintze-Karcher inequality}

\author[S.~Borghini]{Stefano Borghini}
\address{S.~Borghini, Universit\`a degli Studi di Trento,
via Sommarive 14, 38123 Povo (TN), Italy}
\email{stefano.borghini@unitn.it}

\author[M.~Fogagnolo]{Mattia Fogagnolo}
\address{M.~Fogagnolo, Universit\`a di Padova, via Trieste 63, 35121 Padova, Italy}
\email{mattia.fogagnolo@unipd.it}

\author[A.~Pinamonti]{Andrea Pinamonti}
\address{A.~Pinamonti, Universit\`a degli Studi di Trento,
via Sommarive 14, 38123 Povo (TN), Italy}
\email{andrea.pinamonti@unitn.it}
\begin{document}

\begin{abstract}
We provide a rigidity statement for the equality case for the Heintze-Karcher inequality in substatic manifolds. We apply such result in the warped product setting to fully remove  assumption (H4) in the celebrated Brendle's characterization of constant mean curvature hypersurfaces in warped products.
\end{abstract}

\maketitle




\maketitle

\section{Introduction and main statements}
The Heintze-Karcher inequality usually denotes the geometric inequality that, in its more simple form for domains $\Omega$ sitting in $\R^n$, with smooth and strictly mean-convex boundary $\Sigma$, reads as
\begin{equation}
\label{eq:classicalheintze}
    \frac{n-1}{n}\int_\Sigma \frac{f}{\HHH} d \sigma \geq \abs{\Omega}.
\end{equation}
This inequality, that was essentially contained in the seminal earlier paper by Heintze-Karcher \cite{heintze-karcher}  was first pointed out in this very form by Ros \cite[Theorem 1]{ros}, where it was also observed that it holds in a general manifold with nonnegative Ricci curvature. Moreover, he showed that equality in \cref{eq:classicalheintze} is in force only if $\Omega$ is a flat Euclidean balls.

Li-Xia \cite{Li_Xia_17} very vastly generalized Ros' Heintze-Karcher inequality to the setting of \emph{substatic} Riemannian manifolds with horizon boundary. We recall that with this locution we mean a Riemannian manifold $(M, g)$ endowed with a nonnegative smooth function $f$ (the substatic potential) satisfying
\begin{equation}
\label{eq:substatic-intro}
    f \Ric - \nabla\nabla f + \Delta f g \geq 0,
\end{equation}
and where $\partial M = \{f = 0\}$ is a  compact, minimal, regular level set of $f$ (i.e. with ${\nabla f} \neq 0$ on $\partial M$). In particular, the boundary of such a manifold is empty if and only if $f$ is strictly positive. We are occasionally referring to the tensor on the left-hand side of \cref{eq:substatic-intro} as \emph{substatic Ricci tensor}. Very interestingly, it can be in fact checked to arise as Ricci tensor of a suitable affine connection on $(M, g)$, see \cite{Li_Xia_19}. The condition \cref{eq:substatic-intro} stems  naturally from the Einstein Fields Equations of General Relativity, and it is easily observed to hold for initial data sets of \emph{static} spacetimes. We leave the interested reader to~\cite[Lemma 3.8]{Wang_Wang_Zhang} and~\cite[Appendix A]{borghini-fogagnolo} for the explicit computations.

Letting $\Sigma$ be a smooth strictly mean-convex hypersurface homologous to $\partial M$, and $\Omega$ the bounded set enclosed by $\Sigma$ and $\partial M$, the substatic Heintze-Karcher inequality \cite[Theorem 1.3]{Li_Xia_17} has been sharpened
in \cite{fogagnolo-pinamonti} as 
\begin{equation}
\label{eq:HK-intro}
    \frac{n-1}{n}\int_\Sigma \frac{f}{\HHH} d \sigma  \geq \int_\Omega f d \mu + c_{\partial M} \int_{\partial M} \abs{\nabla f} d\sigma,
\end{equation}
where 
\begin{equation}
\label{eq:c-intro}
c_{\partial M} =\frac{n-1}{n} \frac{\int_{\partial M} \abs{\nabla f} d\sigma}{\int_{\partial M}\abs{\nabla f}\left[\frac{\Delta f}{f} - \frac{\nabla \nabla f}{f} \left(\frac{\nabla f}{\abs{\nabla f}}, \frac{\nabla f}{\abs{\nabla f}}\right)\right] d\sigma}\,.
\end{equation}
The above constant has been shown to be well defined and strictly positive in \cite[Proposition 3.2]{fogagnolo-pinamonti}, given the existence of a strictly mean-convex $\Sigma$ homologous to $\pa M$ as in our case.

Our first result shows that a strong rigidity is triggered when \cref{eq:HK-intro} holds with equality sign. 
\begin{theorem}
\label{thm:rigiditybaby}
    Let $(M, g)$ be a substatic Riemannian manifold with connected horizon boundary $\partial M$, such that the substatic potential $f$ satisfies 
    \begin{equation}
    \label{eq:continualbordo-intro}
        \frac{\nabla \nabla f}{f} \in C^{0, \alpha}(M \cup \partial M)
    \end{equation}
    for $\alpha \in (0, 1)$.
    Let $\Sigma$ be a connected, smooth \emph{strictly mean-convex} hypersurface homologous to $\partial M$  Then, the Heintze-Karcher inequality \cref{eq:HK-intro} holds with equality if and only if the domain $\Omega$ such that $\partial \Omega = \Sigma \sqcup \partial M$ is isometric to 
    \begin{equation}
\label{eq:rigidity_intro-homobaby}
       \left([s_0, \overline{s}] \times \partial M, \frac{ds \otimes ds}{f(s)^2} + s^2 g_{\partial M}\right).
    \end{equation}
    \end{theorem}

A version of inequality \cref{eq:HK-intro} was originally obtained by Brendle \cite[Theorems 3.5 and 3.11]{brendle-alexandrov} as the crucial step for obtaining an Alexandrov-type Theorem in warped product manifolds. For proving \cref{eq:HK-intro}, such warped products were assumed to satisfy a set of assumptions (H1)-(H3), recalled and discussed in \cref{subsec:brendle-comparison} below. As we are going to recall, if
they are satisfied, then in particular the warped product is substatic with horizon boundary. 
Thus, \cref{thm:rigiditybaby} directly yields a rigidity statement for Brendle's Heintze-Karcher inequality if the additional, technical \cref{eq:continualbordo-intro} is satisfied. 
Such assumption was fundamental in the elliptic proof of \cref{eq:HK-intro} conceived by Li-Xia \cite{Li_Xia_17} and reworked by the second and third named authors \cite{fogagnolo-pinamonti}. 

\smallskip

Exploiting a new synergy between \cref{thm:rigiditybaby} and the geodesic flow proof worked out in \cite{brendle-alexandrov} to provide \cref{eq:HK-intro} in the warped product setting, we are able to remove assumption \cref{eq:continualbordo-intro} in the special warped product geometry, endowing \cite[Theorems 3.5 and 3.11]{brendle-alexandrov} of an optimal rigidity statement.

\begin{theorem}
\label{thm:warped-intro}
Let $(M, g)$ be a substatic warped product with cross section $\partial M = N$, of the form
\begin{equation}
\label{eq:warpedform-introHK}
    \left([s_0, \overline{s}) \times N, \frac{ds \otimes ds}{f(s)^2} + s^2 g_{N}\right).
\end{equation}
Let $\Sigma$ be a connected, smooth, strictly mean-convex hypersurface homologous to $N$ satisfying \cref{eq:HK-intro} with the equality sign. Then, $\Sigma = \{s = c\}$  for some $c \in (s_0, \overline{s})$.
\end{theorem}
We address the reader to the beginning of \cref{sec:HK-warped} for a more detailed presentation of the very peculiar proof of the above result.

\smallskip

It was already observed in \cite[Section 6]{brendle-alexandrov} that a constant mean curvature hypersurface must fulfil the identity in \cref{eq:HK-intro}, as a consequence of a straightforward Minkowski identity \cite[Proposition 2.3]{brendle-alexandrov}.

 Thus, \cref{thm:warped-intro} directly provides the following characterization of hypersurfaces of constant mean curvature in substatic warped product, improving on \cite[Theorem 1.1]{brendle-alexandrov}. 
\begin{corollary}
\label{cor:alexandrov}
    Let $(M, g)$ be a substatic warped product with a connected horizon boundary, of the form
\begin{equation}
\label{eq:warpedform-intro}
    \left([s_0, \overline{s}) \times N, \frac{ds \otimes ds}{f(s)^2} + s^2 g_{N}\right).
\end{equation}
Let $\Sigma$ be a connected smooth hypersurface homologous to $\partial M = N$ of constant mean curvature. Then, $\Sigma = \{s = c\}$  for some $c \in (s_0, \overline{s})$. 
\end{corollary}
As clarified with additional details in \cref{subsec:brendle-comparison}, the substatic warped products of the form \cref{eq:warpedform-intro} correspond precisely to the family of warped products considered in the Alexandrov-type Theorem \cite[Theorem 1.1]{brendle-alexandrov}. 
On the other hand, such result was proved under an additional assumption, (H4), substantially prescribing the Ricci curvature being smallest in the radial direction. 
While  (H4) is verified on a number of known model solutions, such as the Schwarz\-schild--de Sitter and Reissner--Nordstr\"om manifolds mentioned as applications in~\cite{brendle-alexandrov}, there are important examples where (H4) does not hold. 
Indeed, the Schwarz\-schild--Anti de Sitter warped product
\begin{equation}
\label{eq:antidesitterhyperbolic}
M=[s_0,+\infty)\times N\,,\quad g=\frac{ds\otimes ds}{f^2}+s^2 g_N\,,\quad f=\sqrt{-1+s^2-2ms^{2-n}}\,,
\end{equation}
with cross-section $N$ satisfying $\Ric_{g_N}\geq -(n-2)g_N$ and such that $-\sqrt{(n-2)^{n-2}/{n^n}}<m\leq 0$ is a substatic manifolds with horizon boundary that does not satisfy (H4). In the special case where $\Ric_{g_N}= -(n-2)g_N$, the warped product~\eqref{eq:antidesitterhyperbolic} is a well known vacuum static solution of the Einstein Fields Equations that has been investigated to some extent in the literature (see e.g.~\cite{borghini-negativemass,Chrusciel_Simon} and references therein) and constituted the model for the Lee-Neves Riemannian Penrose Inequality \cite{lee-neves}.
 \cref{cor:alexandrov}, stemming from our novel proof, allows to fully drop the extra hypothesis (H4), hence in particular also applies, for example, to the metric \cref{eq:antidesitterhyperbolic}.

 \subsection*{Further directions and remarks} 
 We conclude mentioning, without any attempt to be complete, a couple of papers where the extra assumption \cref{eq:continualbordo-intro} or (H4) is added in connection with \cite{brendle-alexandrov} and \cite{Li_Xia_17}. Namely, in \cite{scheuer-xia}, the authors provided a quantitative version of Brendle's Alexandrov Theorem by exploiting the alternative proof through elliptic techniques devised in \cite{Li_Xia_17}; consequently they assume \cref{eq:continualbordo-intro}. In \cite{maggi-santilli}, a far-reaching nonsmooth version of Brendle's geodesic flow technique is worked out, leading to a characterization of sets with finite perimeter with a distributional notion of constant mean curvature; as in \cite{brendle-alexandrov}, (H4) is assumed, or some suitable weaker variant (see  \cite[Proposition 3.7]{fogagnolo-pinamonti} and \cite[Remark 1.1]{maggi-santilli}). It may then be fruitful to elaborate on our arguments leading to \cref{thm:warped-intro}, based on the combination of the elliptic and geodesic flow techniques, in order to go beyond \cref{eq:continualbordo-intro} and (H4) also in these kinds of more technical results.  

 \subsection*{Structure of the paper} 
 In \cref{sec:warped} we provide the results of an elementary but, to our knowledge, not yet available analysis of warped product manifolds, and furnish a comparison with Brendle's set of assumptions (H1)-(H4). In \cref{sec:HK-warped} we prove a generalized version of \cref{thm:rigiditybaby}, where hypersurfaces $\Sigma$ with several connected components as well as disconnected horizons are taken into account. In \cref{sec:HK-brendle} we prove \cref{thm:warped-intro}, and deduce \cref{cor:alexandrov}. We conclude the work with an Appendix containing the proofs of the computational results gathered in \cref{sec:warped}.

\subsection*{Acknowledgements} 
This work was initiated during a visit of A. P. at the Centro De Giorgi in Pisa. He warmly thanks the institution for the excellent working conditions offered.
Part of this work has been carried out while S. B. and M. F. were attending the \emph{Thematic Program on Nonsmooth Riemannian and Lorentzian Geometry} that took place at the Fields Institute in Toronto. 
They warmly thank the staff, the organizers and the colleagues for the wonderful atmosphere and the excellent working conditions set up there.

During the preparation of the work, M. F. was supported by the European Union – NextGenerationEU and by the University of Padova under the 2021 STARS Grants@Unipd programme ``QuASAR". 

The authors are members of Gruppo Nazionale per l’Analisi Matematica, la Probabilit\`a e le loro Applicazioni (GNAMPA), which is part of the Istituto
Nazionale di Alta Matematica (INdAM), and are partially funded by the GNAMPA project ``Problemi al bordo e applicazioni geometriche".

A.P. and S.B. are also supported by MIUR and the University of Trento, Italy.

The authors are grateful to Luciano Mari, Lorenzo Mazzieri, Mario Santilli,  Alessandro Savo and Mingxuan Yang 
for the many useful conversations had during the preparation of the work, that substantially helped to improve the quality.

\section{Substatic warped products}
\label{sec:warped}

We consider warped product manifolds
\begin{equation}
\label{eq:generalwarped}
    (I \times N, dr \otimes dr + h^2(r) g_{N}),
\end{equation}
for $(N,g_N)$ a closed $(n-1)$-dimensional Riemannian manifold,
and $I = [0, \overline{r})$, with $\overline{r} > 0$. We will always assume that $\{r=0\}$  is either a horizon boundary or a single point representing the origin of the polar coordinates. Of course, in the latter case, in order for the metric to be smooth at the point with $r=0$, the cross section $(N,g_N)$ must be homothetic to a round sphere. 

When $\dot h\neq 0$, we will find convenient to write warped products also in the equivalent form $(ds\otimes ds)/\dot h(r)^2+s^2 g_N$, for $s=h(r)$. Since we will see in \cref{prop:char-warped} and \cref{lem:splittingdesitter} below that the models we are interested in satisfy $f=\dot h$,
an advantage of this form is that the metric now depends directly on the substatic potential $f$, without the need of the auxiliary warping function $h$. Furthermore, more importantly, the new coordinate $s$ allows to write the function $f$ and the metric $g$ in the explicit form~\eqref{eq:formabuona}.

\begin{proposition}
\label{prop:char-warped}
Let $(M, g)$ be a substatic warped product of the form \cref{eq:generalwarped} with positive nondecreasing $h$, with either a empty boundary or a horizon boundary. 
If the substatic potential $f$ is a function of the coordinate $r$ only and
\begin{equation}
\label{eq:fricciradiale nullo}
    \left[f \Ric - \nabla \nabla f + \Delta f g\right](\nabla r, \nabla r) = 0
\end{equation}
then, up to multiplying $f$ and/or $g_N$ by a positive constant, 
the manifold $(M,g)$ and the substatic potential $f$ satisfy one of the following:
\begin{itemize}

\item[$(i)$] there exists $c\in\R$ such that $f=f(r)$ satisfies $\ddot f+(n-2)c f\geq 0$ and
\begin{equation}
\label{eq:prodotto}
g=dr\otimes dr+\,g_N\,,\quad \Ric_{g_N}\geq (n-2)c g_N\,,
\end{equation}

\item[$(ii)$] there exist $c\in\R$ and a function $\eta:[h(\bar r)^{-n},h(0)^{-n}]\to\R$ with $\eta''\geq 0$ such that
\begin{equation}
    \label{eq:formabuona}
g=\frac{ds\otimes ds}{f^2(s)}+s^2\,g_N\,,\quad \Ric_{g_N}\geq (n-2)c g_N\,,\quad
f\,=\,\sqrt{c+s^2\eta(s^{-n})}
\,.
\end{equation}
In particular, $s=h(r)$ and $f(s)=\dot h(r)$.
\end{itemize}
\end{proposition}

We point out that the family of warped products considered by Brendle \cite{brendle-alexandrov} correspond to the family in $(ii)$ above, see \cref{subsec:brendle-comparison} for more details on this.


In the proof of \cref{thm:warped-intro} we are going to exploit the following strengthening of Proposition~\ref{prop:char-warped}, in force when the substatic Ricci tensor vanishes in an additional direction.
\begin{lemma}
\label{lem:splittingdesitter}
Under the assumptions of \cref{prop:char-warped}, if we also assume that $h$ is not constant and that for every $t\in[a,b]$, $a,b\in\R$ there exist $x\in N$ and a nontrivial $X\in T_xN\subset T_{(t,x)}M$ such that
\begin{equation}
\label{eq:friccitan nullo}
\left[f \Ric - \nabla \nabla f + \Delta f g\right](X, X) = 0\,,
\end{equation}
then, up to multiplying $f$ by a positive constant, in the domain $[a,b]\times N$ the metric $g$ and the function $f$ have the form
\begin{equation}
\label{eq:splittingdesitter}
g=\frac{ds\otimes ds}{f^2}+s^2\,g_N\,,\quad \Ric_{g_{N}}\geq (n-2)c g_N\,,\quad
f\,=\,\sqrt{c-\lambda\,s^2-2m\,s^{2-n}}, 
\end{equation}
    where $f = \dot{h}$ and $\lambda, m \in \R$.
\end{lemma}

The proofs of \cref{prop:char-warped} and \cref{lem:splittingdesitter} involve elementary but lengthy calculations and have been included in the Appendix.

\begin{remark}
The potential $f$ given in \cref{eq:splittingdesitter} coincides with that of the de Sitter/Anti de Sitter--Schwarzschild manifold. When the cross-section is Einstein, these are known to be, {together with cylinders}, the only \emph{static} warped product manifolds with compact horizon boundary, that is, with vanishing substatic Ricci tensor (see~\cite{kobayashi} or~\cite[Section~2.2]{stefano}). \cref{lem:splittingdesitter} constitutes thus a more general warped product classification result.
\end{remark}

It is worth discussing in some detail the regularity of $\eta$ at the horizon boundary. Let $s_0$ be the value of $s$ corresponding to the horizon. We can write $\eta$ in terms of $f$ as $\eta(s^{-n})=s^{-2}(f^2-c)$, hence in particular the value $\eta(s_0^{-n})=-cs_0^{-2}$ at the boundary is finite. We can also show that $\eta'(s_0^{-n})$ is well defined. In fact, we can easily compute 
\[
|\na f|=f(s)f'(s)=s\eta(s^{-n})-\frac{n}{2}s^{1-n}\eta'(s^{-n})\,,
\]
hence the regularity of $\eta'$ up to the boundary follows from the smoothness of $f$ up to the boundary. On the other hand, there does not seem to be an easy way to show the regularity of the second and higher derivatives of $\eta$ up to the boundary. 
This is the very issue that in the end does not allow us to infer that \cref{eq:continualbordo-intro} holds in the warped product case. We are able to show that it holds under the assumption that $\eta$ is $C^{2, \alpha}$ up to the boundary.

\begin{lemma}
\label{lem:desittersoddisfa}
Let $(M, g)$ be a substatic warped product of the form~\eqref{eq:formabuona}. If the function $\eta$ appearing in~\eqref{eq:formabuona}  is $C^{2,\alpha}$ up to the boundary then  $\nabla \nabla f/f \in C^{0,\alpha}(M \cup \partial M) 
$. 
\end{lemma}
\begin{proof}

{

We compute
\begin{equation}
\frac{\nabla \nabla f}{f} = \left(\frac{{f'(s)}^{2}}{f(s)^2} + \frac{f''(s)}{f(s)}\right) ds \otimes ds \,+\, s f(s) f'(s) g_N  = \left[{f'(s)}^{2} + f''(s) f(s)\right] dr \otimes dr + s f(s) f'(s) g_N.
\end{equation}
Moreover, for any function $f$ having the form~\eqref{eq:formabuona}, it holds
\[
sf(s)f'(s)\,=\,s^2\eta(s^{-n})-\frac{n}{2}s^{2-n}\eta'(s^{-n})\,,
\]
\[
{f'(s)}^{2} + f''(s) f(s)\,=\,\eta(s^{-n})+\frac{n(n-3)}{2}s^{-n}\eta'(s^{-n})+\frac{n^2}{2}s^{-2n}\eta''(s^{-n})\,.
\]
Since $s=s_0>0$ at the horizon, from these computations we immediately see how the assumed regularity of $\eta$ implies that of $\nabla \nabla f / f$ up to the boundary $\partial M$. 
}
\end{proof}

\begin{remark}
\label{rem:desittersoddisfaremark}
The warped products in \cref{eq:splittingdesitter} satisfy $\eta(t) = -\lambda -2mt$, hence $\eta'' = 0$. In particular, if such splitting takes place up to the horizon boundary, then \cref{lem:desittersoddisfa} implies that \cref{eq:continualbordo-intro} holds.
\end{remark}

\subsection{Comparison with Brendle's assumptions.}
\label{subsec:brendle-comparison}
In the case of nonempty boundary, in~\cite{brendle-alexandrov}, warped products of the form \cref{eq:generalwarped} with $\Ric_{g_N} \geq (n-2) c g_N$ for some $c\in\R$ are assumed to satisfy the following set of assumptions.
\begin{itemize}
    \item[(H1)] $\dot{h}(0) = 0$, $\ddot{h}(0) > 0$,
    \item[(H2)] $\dot h(r) > 0$ for any $r \in (0, \overline{r})$
    \item[(H3)] The function
    \begin{equation}
\label{eq:h2}
       F(r) =  2 \frac{\ddot{h}(r)}{h(r)} -(n-2)  \frac{c -\dot h(r)^2}{h^2(r)}
    \end{equation}
    is nondecreasing,
    \item[(H4)] It holds
    \[
    \frac{\ddot h(r)}{h(r)}+\frac{c-\dot h(r)^2}{h(r)^2}\,>\,0\,.
    \]
\end{itemize}
Assumptions (H1)-(H3) correspond precisely to case $(ii)$ in \cref{prop:char-warped}. Indeed, since $f = \dot{h}(r)$, assumption (H1) implies that $\{f= 0\}$ coincides with the boundary $\partial M$, that its mean curvature $\HHH=(n-1)\dot h(0)/h(0)$ vanishes and that $|\na f|=\ddot h(0)\neq 0$ at the boundary. In other words, (H1) is equivalent to the request that the boundary is of horizon type.
(H2) entails the request that $f$ is positive in the interior of $M$, while (H3) is instead equivalent to the substaticity of $g$, with substatic potential $f = \dot{h}$, as shown in \cite[Proposition 2.1]{brendle-alexandrov}. It is finally worth pointing out that, remarkably, for a warped product \cref{eq:generalwarped}, equation \cref{eq:fricciradiale nullo} is \emph{always} satisfied for $f = \dot{h}$. This is due to the following formula, contained in the proof of \cite[Proposition 2.1]{brendle-alexandrov}
\begin{equation}
\label{eq:fricciwarped}
f \Ric - \nabla \nabla f + \Delta f g =  f(\Ric_{g_N} - (n-2) c g_N) + \frac{1}{2}\dot{h}(r)^2 \dot F(r) g_N.
  \end{equation}

\smallskip

As pointed out in the introduction, we will never need (H4) in our analysis. For warped products having the form~\eqref{eq:formabuona},  (H4) is equivalent to $sf(s)f'(s)>-c+f(s)^2$. Substituting the formula for $f$ into~\eqref{eq:formabuona} this is in turn equivalent to $\eta'<0$. In particular, the model solutions described in~\eqref{eq:splittingdesitter} satisfy (H4) if and only if $m>0$.


\section{Warped product splitting of substatic manifolds}
\label{sec:HK-warped}
In this section, we are going to prove the following result, more general than \cref{thm:rigiditybaby}, since it deals with possibly disconnected hypersurfaces and explicitly treats the case of components of $\Sigma$ that are homologous to a point. In particular, it fully encompasses the case of an ambient $M$ with empty boundary.
\begin{theorem}
\label{thm:rigiditygeneral}
    Let $(M, g)$ be a substatic Riemannian manifold with possibly empty horizon boundary $\partial M$ such that the substatic potential $f$ satisfies 
    \begin{equation}
    \label{eq:continualbordo}
        \frac{\nabla \nabla f}{f} \in C^{0, \alpha}(M \cup \partial M)
    \end{equation}
    for $\alpha \in (0, 1)$.
    Let $\Sigma$ be a smooth \emph{strictly mean-convex} hypersurface homologous to a possibly empty union $N = N_1 \sqcup \dots \sqcup N_l$ of connected components of $\partial M$. Let $\Sigma_1, \dots, \Sigma_k$, $k \geq l$, the  connected components of $\Sigma$. Assume that, for $1 \leq j \leq l$, each  $\Sigma_j$ is homologous to the component $N_j$ of $\partial M$, while for $j > l$ $\Sigma_j$ is null-homologous. Let $\Omega_j$ be the connected region enclosed by $\Sigma_j$ and $N_j$ if $1 \leq j \leq l$, and the connected region enclosed by $\Sigma_j$ if $l < j \leq k$. Let also $f_j$ the restriction of $f$ on $\Omega_j$.  Then, the Heintze-Karcher inequality \cref{eq:HK-intro} holds with equality if and only  if the following hold.
    \begin{itemize}
  \item[$(i)$] For $1\leq j\leq l$, $(\overline{\Omega_j}, g)$  is isometric to
    \begin{equation}
\label{eq:rigidity_intro-homo}
       \left([s_0^j, s_1^j] \times N_j, \frac{ds \otimes ds}{f_j(s)^2} + s^2 g_{N_j}\right).
    \end{equation}
    \item[$(ii)$] For $l < j \leq k$ , $(\overline{\Omega_j}, g)$ is isometric to
    \begin{equation}
\label{eq:rigidity_intro-etero}
        \left([0, s_1^j] \times N_j, \frac{ds \otimes ds}{f_j(s)^2} + \Big(\frac{s}{f_j(0)}\Big)^2 g_{\Sf^{n-1}}\right).
    \end{equation}
\item[$(iii)$] {We have $f_1(s_0^1)f_1'(s_0^1)/s_0^1=\dots =f_l(s_0^l)f_l'(s_0^l)/s_0^l$.}
\end{itemize}
\end{theorem}
To prove \cref{thm:rigiditygeneral} we start from the full statement of the Heintze-Karcher inequality \cref{eq:HK-intro}, given in \cite[Theorem 3.6]{fogagnolo-pinamonti}. It involves the solution $u$ to the boundary value problem
\begin{equation}
\label{eq:torsion-pb}    
\begin{cases}
\Delta u = - 1 + \frac{\Delta f}{f} u & \mbox{in} \,\, \Omega 
\\
\,\,\,\,\,u=c_N & \mbox{on}\,\, N
\\
\,\,\,\,\,u = 0 &\mbox{on} \,\, \Sigma,
\end{cases}
\end{equation}
where $\partial \Omega = \Sigma \sqcup N$, with $N$ union of connected hypersurfaces $N_1 \sqcup \dots \sqcup N_l$, $l \in \N$ and $c_N$ is the constant given by \cref{eq:c-intro}.
It reads
\begin{equation}
\begin{split}
\frac{n-1}{n} \int_{\Sigma}\frac{f}{\HHH} d\sigma -\int_{\Omega} f d\mu  - \sum_{j \in J} c_j \int_{N_j} \abs{\nabla f} d\sigma
 \, \geq \, \,  \frac{n}{n-1}   \int_\Omega   &\bigg|\nabla \nabla u - \frac{\Delta u}{n} g - u\left(\frac{\nabla \nabla f}{f}  - \frac{\Delta f}{nf} g\right)\bigg\vert^2 \\
 &+ Q\left(\nabla u - \frac{u}{f} \nabla f, \nabla u - \frac{u}{f} \nabla f\right) d\mu,
 \end{split}
 \end{equation}
where
\begin{equation}
    Q = f \Ric - \nabla \nabla f + \Delta f g.
\end{equation}
It yields in particular the following. 
\begin{lemma}
\label{lem:confhess}
    Let $(M, g)$ be a substatic Riemannian manifold with horizon boundary satisfying \cref{eq:continualbordo}, and let $\Sigma$ be a smooth hypersurface homologous to $N = N_1 \sqcup \dots \sqcup \dots N_l$ for $l \in \N$. Let $c_N$ be given by \cref{eq:c-intro}, and let $u$ be the solution to \cref{eq:torsion-pb}. Then, if equality holds in \cref{eq:HK-intro}, then 
    \begin{equation}
        \label{eq:confhess}
        \nabla \nabla u - \frac{\Delta u}{n}g - u \left(\frac{\nabla \nabla f}{f} - \frac{\Delta f}{nf} g\right) = 0 
    \end{equation}
    and 
    \begin{equation}
    \label{eq:friccinunullo}
        [f \Ric - \nabla \nabla f + \Delta f g]\left(\nabla u - \frac{u}{f}\nabla f, \nabla u - \frac{u}{f}\nabla f\right) = 0
    \end{equation}
    in $\overline{\Omega} \setminus N$.
\end{lemma}
The following basic yet fundamental observation provides a conformal warped product splitting for the metric in $\Omega$. It exploits \eqref{eq:confhess} only. In the remainder of this section, \emph{we  agree to define $N_j = \emptyset$ when $j > l$}.
\begin{lemma}
\label{lem:confsplit}
In the assumptions and notations of \cref{thm:rigiditygeneral}, 
let $\phi = u/f$ on $\overline{\Omega} \setminus N$. Then, there exists a coordinate $\rho$ on $\overline{\Omega_j} \setminus N_j$ such that $\phi$ depends on $\rho$ alone and $\overline{\Omega_j} \setminus N_j$ splits as $[0, \overline{\rho}_j) \times \Sigma_j$ endowed with the metric 
\begin{equation}
\label{eq:confsplit}
g_j = f^2_j \left(d\rho \otimes d\rho + \frac{\phi_j^2(\rho)}{f_j^2(0, \theta)} g_{\Sigma_j}\right),
\end{equation}
where $g_{\Sigma_j}$ is the metric induced on $\Sigma_j$ by $g$, $\phi_j$ is the restriction of $\phi$ on $\overline{\Omega_j}\setminus N_j$,  and $\theta = (\theta^1, \dots, \theta^{n-1})$ are coordinates on $\Sigma_j$. Moreover, $\overline{\rho}_j = +\infty$ for $1\leq j \leq l$, and finite for $l< j \leq k$.
\end{lemma}
\begin{proof}
We focus on a single $\overline{\Omega_j} \setminus N_j$, and drop for notational convenience the subscript $j$.

Consider the conformal metric $\tilde{g} = f^{-2} g$. 
Then, it is readily checked that \eqref{eq:confhess} is equivalent to  
\begin{equation}
    \label{eq:confhessgtilde}
    \nabla\nabla_{\tilde{g}} \, \phi - \frac{\Delta_{\tilde{g}} \phi}{n} \, \tilde{g} = 0
\end{equation}
in $\overline{\Omega} \setminus N$. 
We recall moreover that a Hopf Lemma holds for $u$ on $\Sigma$ \cite[Theorem 2.3]{fogagnolo-pinamonti}, and consequently $\Sigma$ is a regular level set for $\phi$. Then,
by classical results (\cite{tashiro}, see otherwise \cite[Section 1]{cheeger-colding-lowerbounds} or \cite{catino-mantegazza-mazzieri}), there exists a coordinate $\rho$ such that $\phi$ is a function of $\rho$ alone and such that $(\overline{\Omega}, \tilde{g})$ splits as $[0, \overline{\rho}) \times \Sigma$ endowed with
\begin{equation}
\label{eq:splittingtildeg}
    \tilde{g} = d\rho \otimes d\rho + \frac{\phi^2(\rho)}{\phi^2(0)} \tilde{g}_{\Sigma},
\end{equation}
with $\overline{\rho}$ being infinite if (and only if) $N$ is nonempty. In \cref{eq:splittingtildeg},  $\tilde{g}_{\Sigma}$ is the metric induced on $\Sigma$ by $\tilde{g}$. This proves \cref{eq:confsplit}.
\end{proof}
Before providing the proof of \cref{thm:rigiditygeneral}, we point out, as another fundamental, although almost trivial, consequence of the assumption \cref{eq:continualbordo}, that $\abs{\nabla f}$ is constant on each connected component of $\partial M$.

\begin{lemma}
\label{lem:gradf-constant}
Let $(M, g)$ be a substatic Riemannian manifold with potential $f$ and nonempty horizon boundary. Assume that $(\nabla \nabla f)/f$ is continuous up to the boundary $\partial M$. Then, $\abs{\nabla f}$ is constant on each connected component of $\partial M$. 
\end{lemma}
\begin{proof}
    Let $N$ be a connected component of $\partial M$, and $X$ any vector field in $TN$. We have
\begin{equation}
       \label{eq:hessfrel}
        \langle \nabla \abs{\nabla f}^2 , X\rangle = 2 \nabla \nabla f(\nabla f, X).
\end{equation}
    on $N$. On the other hand, since $\nabla\nabla f /f$ is continuous up to the boundary $\{f=0\}$, we necessarily have $\nabla \nabla f = 0$ on $N$. By \eqref{eq:hessfrel}, we conclude that $\abs{\nabla f}$ is constant on $N$. 
\end{proof}
In order to complete the proof of \cref{thm:rigiditygeneral}, we are substantially left to prove that the substatic potential $f$ depends on $\rho$ alone. This information, plugged into \cref{eq:confsplit}, implies that $(\overline{\Omega}, g)$ is in fact a warped product, and the conclusion follows from \cref{prop:char-warped}. To achieve this goal, we are going to exploit \cref{eq:friccinunullo}.

\smallskip

\begin{proof}[Proof of \cref{thm:rigiditygeneral}]
Again, we drop the dependency on $j$. We
    consider again  the conformal setting $\tilde{g} = f^{-2} g$.
    Observe that the substatic Ricci tensor just translates into
    \begin{equation}
\label{eq:friccinonconf-conf}
   f\Ric - \nabla \nabla f + \Delta f g  =    f {\Ric}_{\tilde{g}} - (n-1) \nabla\nabla_{\tilde{g}} f + 2(n-1)\frac{df \otimes df}{f}.
    \end{equation}
    Let $T$ be the tensor in the right-hand side above, and observe that substaticity amounts to $T \geq 0$. Moreover, letting again $\phi = u/f$, 
    the condition \cref{eq:friccinunullo}  reads
    \begin{equation}
\label{eq:friccinununulloconf}
       T (\nabla \phi, \nabla \phi) =  (\phi'(\rho))^2 T (\nabla \rho, \nabla \rho) = 0,
    \end{equation}
     where $\nabla \phi =\phi'(\rho) \nabla \rho$ is due to $\phi$ being a function of $\rho$ alone, by \cref{lem:confhess}.

Inspired by the proof of \cite[Lemma 4.6]{benatti2022minkowski}, consider now, for $\theta^i$, $i \in \{1, \dots, n-1\}$ a local coordinate on $\Sigma$ and $\lambda \in \R$, the vector field $Y_i = \nabla \rho + \lambda \partial_i$, where we denoted $\partial_i = \partial /\partial \theta^i$. The condition $T(Y_i, Y_i) \geq 0$, coupled with \eqref{eq:friccinununulloconf}, yields, at any fixed point $p \in \Omega$,
\begin{equation}
\label{eq:lucatrick}
    T(Y_i, Y_i) = 2\lambda T_{i\rho} + \lambda^2 T_{ii} \geq 0
\end{equation}
for any $\lambda \in \R$. This can actually happen only if $T_{i\rho} = 0$. Such condition reads
\begin{equation}
\label{eq:Tjrhonullo}
    (\nabla\nabla)^{\tilde{g}}_{i\rho} f = 2 \partial_\rho f \partial_j f.
\end{equation}
Computing the $\tilde{g}$-Hessian of $f$ from the expression \cref{eq:confsplit}, the above identity becomes
\begin{equation}
    \partial^2_{i\rho} f - \frac{\phi''(\rho)}{\phi'(\rho)} \partial_i f = 2 \frac{\partial_\rho f}{f}.
\end{equation}
One can now directly check that, as a consequence of the above relation, we have
\begin{equation}
    \partial_\rho \left(\frac{\partial_i f}{f^2 \phi'(\rho)}\right) = 0,
\end{equation}
so that, given $0 \leq \rho_1 < \rho_2 < \overline{\rho}$,
we get 
\begin{equation}
\label{eq:partialj1/f}
    \partial_i \left(\frac{1}{f}\right) (\rho_2, \theta) = \frac{\phi'(\rho_2)}{\phi'(\rho_1)} \partial_i \left(\frac{1}{f}\right) (\rho_1, \theta). 
\end{equation}
Integrating both sides along $\theta^i$ in $(\theta^i_0, \theta^i_1)$, and omitting to explicate the dependencies on $\theta^m$ for $m \neq i$, we finally get
\begin{equation}
\label{eq:identitafinale}
\frac{\phi'(\rho_2) f(\rho_2, \theta^i_0) - \phi'(\rho_2)f(\rho_2, \theta^i_1)}{(\phi'(\rho_2))^2f(\rho_2, \theta^i_0)f(\rho_2, \theta^i_1)} = \frac{\phi'(\rho_1) f(\rho_1, \theta^i_0) - \phi'(\rho_1)f(\rho_2, \theta^i_1)}{(\phi'(\rho_1))^2f(\rho_1, \theta^i_0)f(\rho_1, \theta^i_1)}.
\end{equation}
We are now going to show that the left hand side converges to $0$ as $\rho_2 \to \overline{\rho}$. Since we will need to tell between $N_j$ and $N = N_1 \sqcup \dots \sqcup N_l$, we restore the dependencies on $j$ when dealing with a specific connected component $N_j$ of $N$. Recall also that we say that $N_j$ is empty when $j > l$.  

If $N_j$ is empty, then the limit $\rho_2 \to \overline{\rho}$ corresponds to approaching a particular point $p$ in each connected component of $\Omega$, and so the numerator on the left hand side of \cref{eq:identitafinale} converges to zero, while the denominator stays bounded away from zero. 
To understand the case of a nonempty $N$,
recall first that $\phi = u/f$, and observe that $\abs{\nabla \rho} = 1/f$. Then, we have
\begin{equation}
\label{eq:fphi'}
    (f \phi')^2(\rho_2, \theta) = f^4 
    \left(\frac{\abs{\nabla u}^2}{f^2} + \frac{u^2}{f^4} \abs{\nabla f}^2 - 2\frac{u}{f^3} \langle\nabla u, \nabla f\rangle\right)(\rho_2, \theta),
\end{equation}
where all the quantities are understood in terms of $g$. Since $u$ attains smoothly the datum on $N_j$ (see \cite[Theorem 2.3]{fogagnolo-pinamonti}), and $f \to 0$ on $N_j$, that we are approaching as $\rho_2 \to \infty$, we deduce that the limit of the above quantity is given by $c_N^2 \abs{\nabla f}^2_{\mid N_j} > 0$, that crucially is constant by \cref{lem:gradf-constant}. Thus, again, this implies that the left hand side of \cref{eq:identitafinale} vanishes in the limit as $\rho_2 \to \infty$, since so does the numerator, while the denominator tends to $c_N^2 \abs{\nabla f}^2_{\mid N_j} > 0$. 

\smallskip

We conclude, finally, that the numerator of the right hand side of \cref{eq:identitafinale} vanishes for any $\rho_1 \in [0, \overline{\rho})$, implying that $f$ does not depend on $\theta^i$. Being $i \in \{1, \dots, n-1\}$ arbitrary, we deduce that $f$ depends on $\rho$ only. This also implies that $f$ depends on the $g$-distance from $\Sigma$ only, and thus allows to write $g$ on $\Omega$ as a warped product based at $N_j$ if this is nonempty, and on a spherical cross-section otherwise. 
We can thus invoke our characterization of substatic warped products \cref{prop:char-warped}, since \cref{eq:fricciradiale nullo} just amounts to \cref{eq:friccinununulloconf}. 
Moreover, since $f$ is also constant on $\Sigma$, the cylindrical situation \cref{eq:prodotto} cannot arise either, since this would imply that the mean curvature of $\Sigma$ is zero, against the assumption of strict mean-convexity. {We are thus left with \cref{eq:formabuona}. Since the above argument works for any $j$, it follows that every connected component $\Omega_j$ must have the structure prescribed by points $(i)$ and $(ii)$ of the theorem. 

To prove point $(iii)$, we start by observing that all pieces having the form $(i)$ or $(ii)$ also saturate the Heintze-Karcher inequality individually. Imposing equality in~\eqref{eq:HK-intro} in all connected components $\Omega_j$ for $1\leq j\leq l$ as well as equality in the whole domain $\Omega$, we deduce
\begin{equation}
    \label{eq:condcostanti_proof}
    c_N \int_N \abs{\nabla f} d\sigma = \sum_{j=1}^l c_{N_j} \int_{N_j} \abs{\nabla f} d\sigma.
\end{equation}
Using the explicit expression of the metric in the domain $\Omega_j$, provided by point $(ii)$ of the theorem proved above,  we also compute
\[
c_N\,=\,\frac{1}{n}\frac{\sum_{j=1}^l k_j|N_j|}{\sum_{j=1}^l\frac{k_j^2}{s_0^j}|N_j|}\,, \qquad c_{N_j}\,=\,\frac{1}{n}\frac{s_0^j}{k_j}\,,
\]
where $k_j$ is the (constant) value of $|\na f|$ on $N_j=\{s=s_0^j\}\cap\Omega_j$. More explicitly $k_j=f(s_0^j)f'(s_0^j)$. Equation~\eqref{eq:condcostanti_proof} can then be rewritten as
\[
\left(\sum_{j=1}^l k_j|N_j|\right)^2\,=\,\left(\sum_{j=1}^l\frac{k_j^2}{s_0^j}|N_j|\right)\left(\sum_{j=1}^l s_0^j|N_j|\right)\,.
\]
We now show that this equality forces $k_1/s_0^1=\dots =k_l/s_0^l$, concluding the proof of point $(iii)$ of the Theorem. To this end, we actually prove the following more general statement: given positive numbers $\alpha_1,\dots,\alpha_l,\beta_1,\dots,\beta_l$, if it holds
\[
\left(\sum_{j=1}^l\alpha_j\right)^2\,=\,\left(\sum_{j=1}^l\alpha_j\beta_j\right)\left(\sum_{j=1}^l\frac{\alpha_j}{\beta_j}\right)\,,
\]
then $\beta_1=\dots=\beta_l$. A way to show this is via a direct computation: expanding the above terms we have
\[
\sum_{j=1}^l\alpha_j^2+\sum_{i\neq j}\alpha_i\alpha_j\,=\,\sum_{j=1}^l\alpha_j^2+\sum_{i\neq j}\frac{\beta_i}{\beta_j}\alpha_i\alpha_j\,.
\]
Simplifying the equal terms on the left and right hand side and exploiting the symmetry in the indexes $i$ and $j$, the above formula gives
\[
2\sum_{i< j}\alpha_i\alpha_j\,=\,\sum_{i< j}\left(\frac{\beta_i}{\beta_j}+\frac{\beta_j}{\beta_i}\right)\alpha_i\alpha_j\,=\,\sum_{i< j}\frac{\beta_i^2+\beta_j^2}{\beta_i\beta_j}\alpha_i\alpha_j\,.
\]
Since $\beta_i^2+\beta_j^2\geq 2\beta_i\beta_j$, with equality if and only if $\beta_i=\beta_j$, the wished result follows at once.
}
\end{proof}

\section{Heintze-Karcher rigidity in substatic warped products}
\label{sec:HK-brendle}
We start by exploiting Brendle's monotonicity formula to deduce some useful geometric properties along an evolution of a hypersurface fulfilling the identity in \cref{eq:HK-intro}. We focus on the case of a nonempty horizon boundary $\partial M = N$ and of $\Sigma$ homologous to it; the case of $\Sigma$ null-homologous and in particular the case of an ambient $M$ with empty boundary is already fully encompassed by \cref{thm:rigiditygeneral}. 

\smallskip

Consider the conformal metric $\tilde{g} = f^{-2} g$, and let $\Omega_t = \{x \in \Omega \, | \, \rho(\Sigma, x) \geq t\}$, where $\rho$ is the $\tilde{g}$-distance, and $\Omega$ as above is the region enclosed between $\Sigma$ and $\partial M$. Let $\Sigma_t = \partial \Omega_t \setminus \partial M$. Crucially, the mean curvature of $\Sigma_t$ is easily seen to remain strictly positive if the initial $\Sigma$ is strictly mean-convex (see  \cite[Proposition 3.2]{brendle-alexandrov}).
Let 
\begin{equation}
\label{eq:Q(t)}
    Q(t) = \int\limits_{\Sigma_t} \frac{f}{\HHH} d\sigma,
\end{equation}
where all the integrated quantities are expressed in terms of the original metric $g$. Then, utilizing, as in \cite[Proposition 3.2]{brendle-alexandrov},
classical evolution equations (see e.g. \cite[Theorem 3.2]{Huisken_Polden}) for our normal flow of speed $f$, and plugging in the identity
\begin{equation}
\label{eq:laplacian-livelli}
    \Delta_{\Sigma_t} f = \Delta f - \nabla \nabla f (\nu, \nu) -\HHH (\nabla f, \nu)
\end{equation}
(again in terms of $g$), one gets
\begin{equation}
\label{eq:Q'(t)}
    Q'(t) = - \frac{n}{n-1} \int_{\Sigma_t} f^2 d\sigma - \int_{\Sigma_t} \left(\frac{f}{\HHH}\right)^2 \left[\abs{\mathring{\hhh}}^2 +  \left(\Ric - \frac{\nabla\nabla f}{f} +\frac{\Delta f}{f} g\right)(\nu, \nu)\right] d\sigma.
\end{equation}
Let now $t \in (0, \infty)$ be such that $\Sigma_\tau$ is smooth for any $\tau \in [0, t]$. One has, applying the coarea formula, 
\begin{equation}
\label{eq:formula-derivata-integrale}
\begin{split}
Q(0) -Q(t) &= -\int_{0}^t Q'(\tau) d\tau  \\
&= \frac{n}{n-1}\int\limits_{\Omega\setminus\overline{\Omega_t}} f d\mu \, + \int_0^\tau \int_{\Sigma_\tau}\left(\frac{f}{\HHH}\right)^2 \left[\abs{\mathring{\hhh}}^2 +  \left(\Ric - \frac{\nabla\nabla f}{f} +\frac{\Delta f}{f} g\right)(\nu, \nu)\right] d\sigma. 
\end{split}
\end{equation}
As long as $Q(t)$ is smooth, Brendle's Heintze-Karcher inequality \cite[Theorem 3.11]{brendle-alexandrov} states that 
\begin{equation}
\label{eq:HKQ(t)}
Q(t) \geq \frac{n}{n-1}\int_{\Omega_t} f d\mu + c_{N} \int_{\partial M} \abs{\nabla f} d\mu.
\end{equation}
Since equality holds in the Heintze-Karcher inequality for the initial $\Sigma$,
that is 
\begin{equation}
\label{eq:hkequalityQ(0)}
    Q(0) = \frac{n}{n-1}\int_{\Omega} f d\mu + c_{N} \int_{\partial M} \abs{\nabla f} d\mu,
\end{equation}
we get, applying \cref{eq:hkequalityQ(0)} and \cref{eq:HKQ(t)}  to the left hand side of \cref{eq:formula-derivata-integrale},
that
\begin{equation}
\label{eq:totumb+friccinullointegral}
    \int_{\Sigma_\tau}\left(\frac{f}{\HHH}\right)^2 \left[\abs{\mathring{\hhh}}^2 +  \left(\Ric - \frac{\nabla\nabla f}{f} +\frac{\Delta f}{f} g\right)(\nu, \nu)\right] d\sigma = 0
\end{equation}
for any $\tau \in [0, t]$. We deduce the following information on the evolution of $\Sigma$, as long as it remains smooth.
\begin{lemma}
\label{lem:totumb+friccinununullo}
Let $(M, g)$ be a substatic warped product of the form \cref{eq:formabuona}, with a nonempty connected horizon boundary $\partial M$. Let $\Sigma = \partial \Omega \setminus \partial M$ be a smooth, embedded, connected hypersurface homologous to $\partial M$, such that \cref{eq:HK-intro} holds with equality sign. Let $\Sigma_t = \{x \in \Omega \, | \, \rho(\Sigma, x) = t\}$, where $\rho$ is the distance in the conformal metric $\tilde{g} = f^{-2} g$. Then, $\Sigma_t$ is a totally umbilic hypersurface such that
\begin{equation}
\label{eq:friccisigmat}
[f \Ric - \nabla \nabla f + \Delta f g] (\nu, \nu) = 0,
\end{equation}
as long as $\Sigma_t$ evolves smoothly.
\end{lemma}
We now illustrate how we are going to get \cref{thm:warped-intro}. We first show that $\Sigma_t$ remains smooth for all of its evolution, in \cref{prop:smoothevo}. This is fundamentally due to the total umbilicity of the evolution coupled with the Heintze-Karcher inequality itself, preventing the second fundamental form to blow up, see \cref{lem:secondbounded}. 
Then, we adapt to the substatic setting an argument of Montiel \cite{Montiel}, yielding in our case a very peculiar dichotomy: if a totally umbilic hypersurface satisfying \cref{eq:friccisigmat} is not a cross-section of the warped product, then a vector field $X$ tangent to $\partial M$ is found on the region spanned by $\Sigma$ such that the condition \cref{eq:fricciradiale nullo} in \cref{lem:splittingdesitter} is also satisfied (see \cref{prop:montiel}). But then, having showed that $\Omega$ is foliated by such hypersurfaces, this region must split as prescribed by \cref{eq:splittingdesitter}. However, as observed in \cref{rem:desittersoddisfaremark}, this metric satisfies \cref{eq:continualbordo}, and we conclude that the only possibility in the dichotomy is that in fact the initial $\Sigma$ was isometric to a cross-section.

\subsection{The $\tilde{g}$-flow remains smooth.}
In order to show that the second fundamental form does not blow up along a smooth evolution $\Sigma_t$ starting at a hypersurface $\Sigma$ fulfilling the equality in Heintze-Karcher, we first observe that the diameters remain bounded. In this subsection, we are always denoting with $\mathrm{C}_t$ some positive constant possibly depending on $t \in (0, +\infty)$. 
\begin{lemma}
\label{lem:boundeddiam}
In the assumptions of \cref{lem:totumb+friccinununullo}, let $t < +\infty$ be such that $\Sigma_{\tau}$ is smooth for any $\tau \in [0, t)$. 
Then, the metric $g_\tau$ induced by $g$ on $\Sigma_\tau$ satisfies
\begin{equation}
\label{eq:boundg}
|g_\tau| \leq \mathrm{C}_t\,,
\end{equation}
for any $\tau \in [0, t)$
where the norm of $g_\tau$ is induced by the norm of the diffeomorphic surface ${g}_{\Sigma_0}$. Moreover, the intrinsic diameter of $\Sigma_\tau$ satisfies
\begin{equation}
\label{eq:boundediameter}
\mathrm{diam}_{g_\tau}(\Sigma_\tau) \leq \mathrm{C}_t    \end{equation}
for any $\tau \in [0, t)$. 

Both \cref{eq:boundg} and \cref{eq:boundediameter} holds with $g_\tau$ replaced by $\tilde{g}_\tau$, corresponding to the metric induced by the underlying conformal metric $\tilde{g} = f^{-2} g$.
\end{lemma}
\begin{proof}
As long as the flow is smooth, each level $\Sigma_\tau$ is diffeomorphic to $\Sigma=\Sigma_0$. In particular, for all $\tau\in[0,t)$, there exists a metric $g_\tau$ on $\Sigma$ such that $(\Sigma,g_\tau)$ is isometric to $\Sigma_\tau$ endowed with the metric induced on it by $g$. Obviously the same holds for the conformal metrics $\tilde g_\tau$ induced by $\tilde g$.
This allows us to work on a fixed hypersurface $\Sigma$, letting the metrics $g_\tau$, $\tilde g_\tau$ vary in time. 


Notice that $\tilde\HHH=f\HHH-(n-1)\langle\na f\,|\,\nu\rangle>-(n-1)|\na f|\geq -K$, where $K>0$ is the (finite) maximum value of $(n-1)|\na f|$ in $\Omega$, and where we have used that $\HHH$ is strictly positive along the flow by \cite[Proposition 3.2]{brendle-alexandrov}.  By the evolution $\pa_\tau (\tilde g_\tau)_{ij}=-\tilde\HHH\, (\tilde g_\tau)_{ij}$ we have $\pa_\tau\log|(\tilde g_\tau)_{ij}|=-\tilde\HHH< K$, hence
\begin{equation}
\label{eq:boundgconforme}
|(\tilde g_\tau)_{ij}|\,<\,e^{K\tau}|(\tilde g_0)_{ij}|\leq \mathrm{C}_t\,.
\end{equation}
Since $f$ is bounded in the compact domain $\Omega$ enclosed by $\Sigma$, the above bound implies a fully equivalent one in terms of the metric $g_\tau$ induced by $g$ on $\Sigma_\tau$. This proves~\eqref{eq:boundg}.

For a fixed $\tau\in[0,t)$, let $x_\tau,y_\tau\in\Sigma$ be two points realizing the diameter $\mathrm{diam}_{g_\tau}(\Sigma)$ that we want to estimate, and let $\gamma:[0,\ell]\to \Sigma$ be the $ g_0$-unit length geodesic minimizing the distance between $x_\tau$ and $y_\tau$, with respect to the starting conformal metric $g_0$. 

By \cref{eq:boundg}, the length of $\gamma$ is directly estimated as follows:
\begin{align}
|\gamma|_{g_\tau}\,=\,\int_0^\ell |\dot\gamma(s)|_{g_\tau}ds
\,=\,
\int_0^\ell \sqrt{(g_\tau)_{ij}\dot\gamma^i\dot\gamma^j}(\gamma(s))ds \leq \ell \, \mathrm{C}_t 
\end{align}
By construction, the diameter $\mathrm{diam}_{g_\tau}(\Sigma)$ coincides with the $g_\tau$-distance between the endpoints of $\gamma$, so $\mathrm{diam}_{g_\tau}(\Sigma)$ must be less than or equal to the $g_\tau$-length of $\gamma$. Moreover, by construction, we have $\ell\leq {\rm diam}_{ g_0}(\Sigma)$, and so  we have shown
\[
\mathrm{diam}_{g_\tau}(\Sigma) \leq\abs{\gamma}_{g_\tau} \leq \mathrm{C}_t.
\]
This provides the desired uniform bound on the diameter $\mathrm{diam}_{g_\tau}(\Sigma)$.
\end{proof}
The following is the main observation triggering the smooth long time existence along the $\tilde{g}$-distance flow. 
\begin{lemma}
\label{lem:secondbounded}
In the assumptions of \cref{lem:totumb+friccinununullo}, let $t$ be such that $\Sigma_{\tau}$ is smooth for any $\tau \in [0, t)$. Then, the second fundamental form $\hhh_\tau$ of $\Sigma_\tau$ satisfies
\begin{equation}
\label{eq:secondbounded}
    \abs{\hhh_\tau} \leq \mathrm{C}_t
\end{equation}
for any $\tau \in [0, t)$.
\end{lemma}
\begin{proof}
Assume by contradiction that there exists a sequence $\tau_j \to t < +\infty$ as $j \to +\infty$ and points $x_{\tau_j} \in \Sigma_{\tau_j}$ such that $\abs{\hhh_{\tau_j}}(x_{\tau_j})$ blows up as $j \to +\infty$. Then, since the $\Sigma_\tau$'s are totally umbilical, the mean curvature $\HHH_{\tau_j}(x_{\tau_j})$ blows up too. We first show that, then, 
\begin{equation}
\label{eq:meancurvatureexplodesuniformly}
\inf_{x \in \Sigma_{\tau_j}}\HHH_{\tau_j} (x) \to +\infty.
\end{equation}
Indeed, by the Gauss-Codazzi equations and exploiting the total umbilicity we immediately get
\begin{equation}
\label{eq:gausscodazziapplied}
    \nabla_i \HHH(\tau_j) = -\frac{n-2}{n-1}\Ric_{i\nu}
\end{equation}
for any $i \in\{1, \dots, n-1\}$. Since the right hand side is uniformly bounded in $\Omega$, we deduce that $\nabla \HHH$ is uniformly bounded along the evolution. Let then $x \in \Sigma_{\tau_j}$ different from $x_{\tau_j}$. We have
\begin{equation}
\label{eq:unifdiamexploited}
    \HHH(x) \geq \HHH(x_{\tau_j}) - \mathrm{diam}(\Sigma_{\tau_j})\sup_{y\in \Sigma_{\tau_j}}\abs{\nabla \HHH}(y)  \geq \HHH(x_{\tau_j}) - \mathrm{C}_t,
\end{equation}
where the bound on the diameter is  \cref{eq:boundediameter};  \cref{eq:meancurvatureexplodesuniformly} follows. On the other hand, recall that by the Heintze-Karcher inequality 
we have
\begin{equation}
\label{eq:hkapplied}
    \int\limits_{\Sigma_{\tau_j}}\frac{f}{\HHH} d\sigma \geq c_{\partial M} \int_{\partial M} \abs{\nabla f} d\sigma.
\end{equation}
Now, the evolution equations for $\Sigma_\tau$ imply that $\abs{\Sigma_\tau} \leq \abs{\Sigma}$, while, since $t<+\infty$ and $\rho(\Sigma, \partial M) = +\infty$, $\sup_{\Sigma_{\tau_j}} f \leq \mathrm{C}_t$ for some finite $\mathrm{C}_t > 0$. Exploiting this information, we get at once from \cref{eq:hkapplied} that
\begin{equation}
\label{eq:conclusioninfH}
    \inf_{x \in \Sigma_{\tau_j}} \HHH(x) \leq \mathrm{C}_t \abs{\Sigma} \, \left(c_{\partial M} \int_{\partial M} \abs{\nabla f} d\sigma\right)^{-1}, 
\end{equation}
yielding a contradiction with \cref{eq:meancurvatureexplodesuniformly} that completes the proof.
\end{proof}
Concluding from the above that $\Sigma_t$ remains smooth for any $t \in (0, +\infty)$ turns out to be slightly technical, but very classical in nature. The arguments employed were  pioneered by Hamilton \cite{hamilton} in an intrinsic flow setting, and adapted to extrinsic flows by Huisken \cite{huisken-meancurvatureflow}.
\begin{proposition}
\label{prop:smoothevo}
Let $(M, g)$ be a substatic warped product of the form \cref{eq:formabuona}, with a nonempty connected horizon boundary $\partial M$. Let $\Sigma = \partial \Omega \setminus \partial M$ be a smooth, embedded, connected hypersurface homologous to $\partial M$, such that \cref{eq:HK-intro} holds with equality sign. Let $\Sigma_t = \{x \in \Omega \, | \, \rho(\Sigma, x) = t\}$, where $\rho$ is the distance in the conformal metric $\tilde{g} = f^{-2} g$. Then, $\Sigma_t$ is a smooth, embedded, totally umbilic hypersurface such that
\begin{equation}
\label{eq:friccisigmat_2}
[f \Ric - \nabla \nabla f + \Delta f g] (\nu, \nu) = 0,
\end{equation}
for any $t \in [0, \infty)$.
\end{proposition}
\begin{proof}
We are going to show that the nonempty set $T\subseteq [0, +\infty)$ defined by
\begin{equation}
\label{eq:setT}
    T = \{t\in [0, +\infty) \, | \, \Sigma_\tau \, \text{is smooth and embedded for} \, \tau \in [0, t]\}
\end{equation}
is both open and closed in $[0, +\infty)$, inferring the eternal smoothness of the flow. The identity \cref{eq:friccisigmat_2} is then a direct consequence of \cref{lem:totumb+friccinununullo}. 

The openness of $T$ is well-known in general; if a closed hypersurface $\Sigma$ is smooth and embedded, then so are the equidistant hypersurfaces $\Sigma_r = \{x \in M \, | \, \mathrm{dist}(\Sigma, x) = r\}$ for any Riemannian metric-induced distance $\mathrm{dist}$, see e.g. \cite[Proposition 5.17]{mantegazza-notesdistance}. In our case, such result is applied to $\Sigma_t$ with $t \in T$ and with respect to the distance induced by $\tilde{g}$.

The closedness of $T$ constitutes the bulk of the proposition, and will be substantially ruled by \cref{lem:secondbounded} only. We are repeatedly employing the evolution equations for $\Sigma_\tau$ along the $\tilde{g}$-distance flow, in the conformal background metric $\tilde{g}$. Indeed, in this setting such equations are simpler to handle, and, since we are staying away from $\partial M = \{f = 0\}$ the estimates we are inferring will automatically hold also in terms of $g$, and \emph{viceversa}.   
In the remainder of this proof, all the quantities taken into account are thus understood as referred to $\tilde{g}$, even when not explicitly pointed out.

Let $T \ni t_j \to t^{-}$ as $j \to +\infty$. Then, $\Sigma_\tau$ is smooth and embedded for any $\tau\in [0, t)$ We want to show that $\Sigma_t$ is smooth and embedded.  
To accomplish this task, we are going to show that 
\begin{equation}
\label{eq:claimboundnablah}
\abs{\nabla^{(k)}{\hhh}_{{\tau}}} \leq \mathrm{C}_t
\end{equation}
for an arbitrary $k \in \N$,
where $\nabla^{(k)}$ is the $k$-th covariant derivative induced by $\tilde{g}$ on $\Sigma_\tau$. Indeed, if this holds, then all the derivatives of the functions whose graphs describe $\Sigma_\tau$ would be uniformly bounded as $\tau \to t$, implying that $\Sigma_t$ would be actually smooth. 
We are going to prove \cref{eq:claimboundnablah} by induction. The case $k = 0$ corresponds to the \cref{eq:secondbounded}, and we assume 
\begin{equation}
\label{eq:claimboundnablah-1}
\abs{\nabla^{(l)}{\hhh}_{{\tau}}} \leq \mathrm{C}_t
\end{equation}
holds for any $l \in \{0, \dots, k-1\}$. We employ the concise notation $T * Q$ to indicate, at some fixed point, linear combinations of contractions of a tensor $T$ with a tensor $Q$ through the metric tensor. The uniform bound on the evolving metric tensors $g_\tau$ was observed in \cref{eq:boundg}. We have
\begin{equation}
\label{eq:derivativenablakh}
    \frac{\partial}{\partial \tau} \nabla^{(k)} \hhh_\tau = \nabla^{(k)} \frac{\partial}{\partial \tau} \hhh_\tau + \nabla^{(l_1)}\frac{\partial}{\partial \tau}\Gamma*\nabla^{(l_2)}\hhh_\tau,
\end{equation}
where $l_1, l_2 \in \N$ satisfy $l_1 + l_2 = k-1$, and the components of $\Gamma$ are the Christoffel symbols of the evolving metric that $\tilde{g}$ induces on $\Sigma_\tau$. We recall that the variation of the components of $\Gamma$ are in fact components of a  tensor, and that it holds
\begin{equation}
\label{eq:derchristoffel}
    \frac{\partial}{\partial \tau} \Gamma^i_{jm} = \frac{1}{2}g^{ir} \left(\nabla_j\frac{\partial}{\partial \tau}g_{mr} + \nabla_m \frac{\partial}{\partial \tau} g_{jr} - \nabla_r \frac{\partial}{\partial \tau} g_{jm} \right),
\end{equation}
for $i, j, l, r \in\{1, \dots, n-1\}$ and where we meant with $g$ the metric induced by $\tilde g$ on $\Sigma_\tau$. Moreover, the second fundamental from $\hhh_\tau$ induced by $\tilde{g}$ on $\Sigma_\tau$ roughly evolves by (see e.g. \cite[Theorem 3.2, 4]{Huisken_Polden})
\begin{equation}
\label{eq:evolutionsecondfundamental}
    \frac{\partial}{\partial \tau} \hhh_\tau = \hhh_\tau * \hhh_\tau + \mathrm{Riem},
\end{equation}
where $\mathrm{Riem}$ denotes some component of the Riemann tensor of the ambient $\tilde{g}$. Observe that, since $t$ is finite, $\mathrm{Riem}$, as well as any of its $\tilde{g}$-covariant derivative, remains bounded on $\Sigma_\tau$ as $\tau \to t^{-}$. Plugging \cref{eq:evolutionsecondfundamental} and \cref{eq:derchristoffel} into \cref{eq:derivativenablakh}, and directly estimating by means of \cref{eq:claimboundnablah-1}, we get that
\begin{equation}
\label{eq:derivativenablahestimated}
    \frac{\partial}{\partial \tau} \nabla^{(k)} \hhh_\tau = \nabla^{(k)}\hhh_\tau * T + Q,
\end{equation}
where $T$ and $Q$ are tensors uniformly bounded on $\Sigma_\tau$ also as $\tau \to t^-$. Then, taking into account once again that $\partial_\tau g_\tau = -\HHH_\tau g_\tau$, and that is an uniform bounded quantity as $\tau \to t^-$ thanks to \cref{eq:secondbounded}, we deduce from \cref{eq:derivativenablahestimated}
\begin{equation}
\label{eq:boundderivativenormderh}
    \frac{\partial}{\partial \tau} \abs{\nabla^{(k)} \hhh_\tau}^2 \leq \mathrm{C_1} \abs{\nabla^{(k)} \hhh_\tau}^2 + \mathrm{C}_2,
\end{equation}
where both $\mathrm{C}_1$ and $\mathrm{C}_2$ are constants uniformly bounded as $\tau \to t^{-}$. Integrating \cref{eq:boundderivativenormderh} for $\tau \in [0, t)$ provides the claimed \cref{eq:claimboundnablah}, inferring the smoothness of $\Sigma_t$.

\smallskip

{We are left to discuss the embeddedness of $\Sigma_t$. Since $\Sigma_t$ is compact, it is sufficient to show that $\Sigma_t$ has no self intersections. Without loss of generality, suppose that $t$ is the first time such that $\Sigma_{t}$ has a self-intersection $x$. Then, in a neighborhood of $x$, by smoothness and compactness, $\Sigma_t$ is a finite union of smooth embedded hypersurfaces $S_1,\dots, S_k$. 
First of all, if two of these hypersurfaces, say $S_1$ and $S_2$, have different tangent spaces, then it is easily seen by continuity of the flow that $\Sigma_{t-\ep}$ also has self-intersections for times $t-\ep$ close to $t$ (see Figure~\ref{fig:case1}), against our assumption that $t$ is the first value having them.

Thus, it only remains to analyze the case where the tangent space to $S_1,\dots,S_k$ is the same. If the outward pointing normal is the same for $S_1$ and $S_2$, then it is easy to see that the level set $\Sigma_{t-\ep}$ must intersect $\Sigma_t$ (this case is exemplified in Figure~\ref{fig:case2}). This is impossible since $\Sigma_t$ must be contained in the interior of the compact domain $\Omega_{t-\ep}$ enclosed by $\Sigma_{t-\ep}$ for $\ep>0$.

The only case that is left to rule out is the case where there are exactly two hypersurfaces $S_1$ and $S_2$, with outward normals pointing in opposite directions. With similar reasonings as in the previous cases, we can conclude that the situation is as in Figure~\ref{fig:case3}. Namely, we can find coordinates $(x^1,\dots,x^n)$  centered at $x$ such that $S_1=\{x^1= 0\}$ and its outward normal points towards $\{x^1\leq 0\}$, whereas $S_2$ is contained in $\{x^1\geq 0\}$ and its outward normal points towards $\{x^1\geq 0\}$. Recalling that $\Sigma_t$ is mean convex, this configuration is clearly ruled out by the maximum principle.}
\end{proof}


\begin{figure}
 \centering
 \subfigure[Linearly independent normals\label{fig:case1}]
   {\includegraphics[scale=1.3]{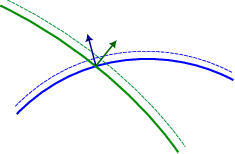}}
 \hspace{10mm}
 \subfigure[Same normal\label{fig:case2}]
   {\includegraphics[scale=1.3]{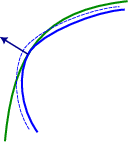}}
   \hspace{10mm}
 \subfigure[Opposite normals\label{fig:case3}]{\includegraphics[scale=1.3]{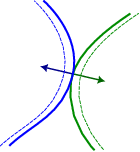}}
 \caption{\small Different configurations at the self-intersection. In blue and green are the two hypersurfaces $S_1$, $S_2$ and the corresponding normal vectors. The dashed lines represent the backward evolution of $S_1$ and $S_2$, that is, they are contained in $\Sigma_{t-\ep}$, for a small $\ep>0$.
 \label{fig:configs}}
 \end{figure}

\subsection{Montiel-type argument and conclusion.}

Thanks to~\cref{eq:fricciradiale nullo} and~\cref{prop:smoothevo}, we have two directions along which the tensor $f\Ric-\na\na f+\Delta f g$ vanishes. A crucial observation is that, if $\Sigma$ is not a cross-section, these two directions are distinct at almost all points of $\Sigma$. This follows from an argument of Montiel \cite{Montiel}. In our substatic setting, thanks to \cref{lem:desittersoddisfa}, this forces the region $\Sigma$ lives in to be of the special form \cref{eq:splittingdesitter}. 
\begin{proposition}
\label{prop:montiel}
Let $(M, g)$ be a substatic warped product of the form~\cref{eq:formabuona}, with a nonempty connected horizon boundary $\partial M$. Let $\Sigma = \partial \Omega \setminus \partial M$ be a smooth, embedded, orientable, connected hypersurface homologous to $\partial M$. Suppose that $\Sigma$ is totally umbilical,  that it holds
\[
\left[f \Ric - \nabla \nabla f + \Delta f g\right](\nu,\nu) = 0\,
\]
and that $\Sigma$ is \emph{not} a cross-section.
where $\nu$ is a unit normal to $\Sigma$.
Then, the function $f$ has the form~\eqref{eq:splittingdesitter} in the region $[s_{\rm min},s_{\rm max}]\times N$, where
\begin{equation}
    s_{\rm min} = \min\{s(x), x \in \Sigma\}, \qquad \qquad
    s_{\rm max} = \max\{s(x), x \in \Sigma\}.
\end{equation}
\end{proposition}

\begin{proof}
The function $f$ depends on the coordinate $s$ only, so it is well defined (up to a constant) the function $\phi=\int (s/f) ds$. We can compute rather easily $\na\na\phi=f g$. In other words, the vector $Y=\na\phi=sf\pa/\pa s$ satisfies $\na Y =\na\na\phi=fg$. Let $Y^\top=Y-g(Y,\nu)\nu$ be the projection of $Y$ on $\Sigma$. If $\nabla_\Sigma$ is the covariant derivative induced by $\nabla$ on $\Sigma$, then $Y^\top=\nabla_\Sigma\phi$. 

We are assuming that $\Sigma$ is not a cross section, that is, $Y^\top$ does not vanish pointwise on $\Sigma$. Following the argument in~\cite[Lemma~4]{Montiel}, for every vector field $Z\in \Gamma(T\Sigma)$ we compute
\begin{align}
\na_\Sigma\na_\Sigma\phi(Z,\cdot)&=\nabla_\Sigma Y^\top(Z,\cdot)
\\
&=\left[\nabla_Z(Y-g(Y,\nu)\nu)\right]^{\top}
\\
&=(\nabla_Z Y)^\top-\left[\nabla_Z g(Y,\nu)\nu+g(Y,\nu)\nabla_Z\nu\right]^\top,
\end{align}
where we are using the notation ${}^\top$ to denote the projection on $\Sigma$ (namely, for a vector field $X\in \Gamma(TM)$, we denote by $X^\top$ the vector $X-g(X,\nu)\nu\in \Gamma(T\Sigma)$). 
Since $\nabla Y=fg$, $\nu^\top=0$, and $(\nabla_Z\nu)^\top=\hhh(Z,\cdot)$, where $\hhh$ is the second fundamental form of $\Sigma$, we deduce
\[
\nabla_\Sigma\nabla_\Sigma\phi=fg_\Sigma-g(Y,\nu)\hhh.
\]
In particular, if $\Sigma$ is umbilical, then $\hhh=\HHH/(n-1)g_\Sigma$ and we obtain
\[
\nabla_\Sigma\na_\Sigma\phi=\left[f-g(Y,\nu)\frac{\HHH}{n-1}\right]g_\Sigma.
\]
Since $\Sigma$ is compact and $Y^\top$ does not vanish pointwise (meaning that $\phi$ is nontrivial), it is then well known \cite{tashiro, cheeger-colding-lowerbounds, catino-mantegazza-mazzieri} that $\Sigma$ must be a warped product with spherical cross-sections. Namely
\[
\Sigma=[0,R]\times\mathbb{S}^{n-2},\qquad g_\Sigma=d\rho\otimes d\rho+\lambda^2g_{\mathbb{S}^{n-2}},
\]
where $\rho$ is the coordinate on $[0,R]$ and $\lambda=\lambda(\rho)$ is positive in $(0,R)$, $\lambda(0)=\lambda(R)=0$, $\lambda'(0)=-\lambda'(R)=1$.
Furthermore the following relations hold:
\[
\lambda'=f-g(Y,\nu)\frac{\HHH}{n-1},\qquad \lambda=\frac{\pa}{\pa\rho}\phi_{|_{\Sigma}},\qquad Y^\top=\nabla_\Sigma\phi=\lambda\frac{\pa}{\pa\rho}.
\]
In particular, $Y^\top$ is different from zero on $\Sigma \setminus \{x, y\}$, with $x, y$ being  the two points corresponding to $\rho=0$ and $\rho=R$. 

Since we are assuming that the tensor $f \Ric - \nabla \nabla f + \Delta f g$ vanishes in the $\nu$ direction and we know this holds also in the $Y$ direction (recall \cref{subsec:brendle-comparison}), we must also have
\[
\left[f \Ric - \nabla \nabla f + \Delta f g\right](Y^\top,Y^\top)=0
\]
at all points of $\Sigma$.
We can then apply \cref{lem:splittingdesitter} with $X = Y^{\top}$ 
to conclude.
\end{proof}

\begin{remark}
The above proof intriguingly shows also that a hypersurface $\Sigma$ as in the statement of \cref{prop:montiel} is itself a warped product.
\end{remark}
We are now ready to conclude the proof of \cref{thm:warped-intro}.
\begin{proof}[Proof of \cref{thm:warped-intro}]

We can restrict our attention to the case of nonempty boundary $\partial M = \{f =0\}$ with $\Sigma$ homologous to $\partial M$, since the empty boundary (or null-homologous) case is fully covered by \cref{thm:rigiditygeneral}. 
We consider the evolution of $\Sigma$ given by the $\Sigma_t \subset \Omega$ at $\tilde{g}$-distance $t$. By \cref{prop:smoothevo}, $\Sigma_t$ is smooth for any $t \in [0, +\infty)$. Suppose by contradiction that $\Sigma$ is not a cross section. Then, $\Sigma_t$ is not a cross-section for any $t \in (0, +\infty)$, for otherwise all of its $\tilde{g}$-equidistant hypersurfaces would be cross-sections, including $\Sigma$. Moreover, as recalled in \cref{lem:totumb+friccinununullo}, the $\Sigma_t$'s are totally umbilical  
and satisfy
\[
\left[f \Ric - \nabla \nabla f + \Delta f g\right](\nu,\nu)=0\,.
\]
We can then apply \cref{prop:montiel} to any $\Sigma_t$, and deduce that in the region foliated by such evolution $f$ can be written as \eqref{eq:splittingdesitter}. 
Since, as $t \to +\infty$, $\Sigma_t$ by construction gets closer and closer to $\partial M$, this holds in the whole of $\Omega$. But then, as observed in \cref{rem:desittersoddisfaremark}, by \cref{lem:desittersoddisfa} the condition \cref{eq:continualbordo} is satisfied.
We can then apply \cref{thm:rigiditybaby} to conclude that $\Sigma$ is a cross-section, a contradiction that concludes the proof.
\end{proof}

\appendix

\section*{Appendix: Proof of Proposition~\ref{prop:char-warped} and Lemma~\ref{lem:splittingdesitter}}

\setcounter{equation}{0}
\renewcommand{\theequation}{A\thesection.\arabic{equation}}
\renewcommand{\thesubsection}{A\thesection.\arabic{subsection}}

We consider warped products 
\[
M=I\times N\,,\qquad g=dr\otimes dr+h^2g_N\,,
\]
with $h=h(r)$ positive, satisfying the substatic condition
\begin{equation}
\label{eq:substatic_cond_app}
\Ric-\frac{\nabla\nabla f}{f}+\frac{\Delta f}{f}g\geq 0
\end{equation}
for some function $f=f(r)$ that is assumed to be nonnegative and zero exactly on the (possibly empty) boundary of $M$. Our aim will be that of proving Proposition~\ref{prop:char-warped}.  To this end, we start by writing down the components of the relevant quantities in the substatic condition.

The Ricci tensor of a warped product is known to satisfy
\begin{align}
\RRR_{rr}\,&=\,-(n-1)\frac{\ddot h}{h}\,,
\\
\RRR_{ir}\,&=\,0\,,
\\
\RRR_{ij}\,&=\,\RRR^N_{ij}-\left[h\ddot h+(n-2)\dot h^2\right]g^N_{ij}\,.
\end{align}
Since both $f$ and $h$ are functions of the coordinate $r$ only, the Hessian and Laplacian are given by the following formulas
\begin{align}
\na^2_{rr}f\,&=\,\ddot f\,,
\\
\na^2_{ir}f\,&=\,0\,,
\\
\na^2_{ij}f\,&=\,h\dot h\dot f g^N_{ij}\,,
\\
\Delta f\,&=\,\ddot f+(n-1)\frac{\dot h}{h}\dot f\,.
\end{align}
Substituting in~\eqref{eq:substatic_cond_app}, we find out that the substatic condition is equivalent to the following two inequalities:
\begin{equation}
\label{eq:substatic_wp}
\begin{aligned}
\dot h\frac{\dot f}{f}\,&\geq\,\ddot h\,,
\\
\Ric_{g_N}\,&\geq\,
h^2\left[\frac{\ddot h}{h}-\frac{\ddot f}{f}+(n-2)\frac{\dot h^2}{h^2}-(n-2)\frac{\dot h}{h}\frac{\dot f}{f}\right]g_N\,.
\end{aligned}
\end{equation}

We are now ready to provide the proof of \cref{prop:char-warped}, telling between the case of a constant $h$, that is the cylindrical splitting of, and of a nonconstant $h$.

\begin{proof}[Proof of \cref{prop:char-warped}]

\emph{The product case.}
We impose $\dot h=0$. Up to a rescaling of $g_N$ we can then just set $h\equiv 1$. The first identity in~\eqref{eq:substatic_wp} is trivial when $\dot h=0$.
The second inequality in~\eqref{eq:substatic_wp} instead reduces to
\[
\Ric_{g_N}\,\geq\,
-\frac{\ddot f}{f}g_N\,.
\]
In particular, if $c$ is the minimum value such that there exists $X\in TN$ with $\Ric_{g_N}(X,X)=(n-2)c |X|^2_{g_N}$, we must have
\[
\ddot f+(n-2)cf\,\geq\,0\,.
\]
For any $f$ satisfying the above inequality and any $(n-1)$-dimensional Riemannian manifold $(N,g_N)$ with $\Ric_{g_N}\geq (n-2)cg_N$ the product manifold $(I\times N,dr\otimes dr+g_N)$ is substatic. 


\smallskip

\noindent
\emph{The warped product case.}
We now consider the case where $h$ is not constant. Since we are assuming
\[
[f \Ric - \nabla \nabla f + \Delta f g] (\na r, \na r) = 0,
\]
the first inequality in~\eqref{eq:substatic_wp} is saturated, forcing $f=k\dot h$, for some constant $k\in\R$. Letting also $c\in\R$ be the minimum value such that there exists $X\in TN$ with $\Ric_{g_N}(X,X)=(n-2)c |X|^2_{g_N}$, we can rewrite the second inequality in~\eqref{eq:substatic_wp} as follows
\begin{equation}
\label{eq:substatic_casef=hdot_wp}
\frac{\dddot h}{\dot h}+(n-3)\frac{\ddot h}{h}-(n-2)\frac{\dot h^2-c}{h^2}\,\geq\,0\,.
\end{equation}

This inequality also appears in~\cite[p.~253]{brendle-alexandrov}.
Since $f$ is assumed to be positive in $M$, notice that in particular this forces $\dot h$ to have a sign. Up to changing the sign of the coordinate $r$, we can assume $\dot h>0$. In particular, $h$ is a monotonic function of $r$, which means that we can use $h$ as coordinate in place of $r$. We use $'$ to denote derivative with respect to $h$. 
Considering then the function 
\[
\psi\,=\,\frac{\dot h^2-c}{h^2}\,,
\]
observing that $\psi'=\dot \psi/\dot h$, we compute
\begin{align}
\left(h^{n+1}\psi'\right)'\,&=\,2\left(h^{n-1}\ddot h-h^{n-2}(\dot h^2-c)\right)'
\\
&=\,2h^{n-1}\left(\frac{\dddot h}{\dot h}+(n-3)\frac{\ddot h}{h}-(n-2)\frac{\dot h^2-c}{h^2}\right)\,.
\end{align}
Therefore,
inequality~\eqref{eq:substatic_casef=hdot_wp} gives
\begin{equation}
\label{eq:substatic_wp_radial_simplified}
\left(h^{n+1}\psi'\right)'\,\geq\,0\,,
\end{equation}
which in turn tells us that
\[
\psi\,=\,\int\frac{\mu}{h^{n+1}}dh\,,
\]
where $\mu=\mu(h)$ satisfies $\mu'\geq 0$. 
This is equivalent to asking $\psi=\eta(h^{-n})$, with $\eta''\geq 0$.
More explicitly, the substatic potential $f$ and the warping function $h$ must satisfy
\[
f\,=\,k\dot h\,=\,k\sqrt{c+h^2\eta(h^{-n})}\,,\quad \eta''\geq 0\,.
\]
Summing all up, we have found that all substatic warped products $(I\times N,g,f)$ with $f$ radial and such that $f \Ric - \nabla \nabla f + \Delta f g$ vanishes in the radial direction are isometric to a solution having the following form
\[
M=[a,b]\times N\,,\quad g=k^2\frac{ds\otimes ds}{f^2}+s^2\,g_N\,,\quad \Ric_{g_N}\geq (n-2)c g_N\,,\quad
f\,=\,k\sqrt{c+s^2\eta(s^{-n})}\,,
\]
where $0<a<b$, $c>0$, $k>0$ are constants, $s$ is a coordinate on $[a,b]$, and $\eta:[b^{-n},a^{-n}]\to\R$ satisfies $\eta''\geq 0$.
\end{proof}
We conclude with the proof of \cref{lem:splittingdesitter}.
\begin{proof}[Proof of \cref{lem:splittingdesitter}]
If we further assume, as in the statement, that at the point $p$ with coordinates $(s_0,x)$ there exists a vector $X\in T_pN$ such that $[f \Ric - \nabla \nabla f + \Delta f g](X,X)=0$, then it easily follows that the second inequality in~\eqref{eq:substatic_wp} is saturated. Retracing the computations above one then deduces that $\eta''=0$ at $s=s_0$. If it is possible to find such a vector $X$ for every $s$ in an interval $[s_0,s_1]\subset [a,b]$, then we can integrate the identity $\eta''=0$ in $[s_1^{-n},s_0^{-n}]$, obtaining $\eta(t)=-\lambda-2mt$ for constants $m,\lambda$. Substituting in the formula for $f$ we then obtain
\[
f\,=\,k\sqrt{c-\lambda\,s^2-2m\,s^{2-n}}\,,
\]
as claimed.
\end{proof}

\printbibliography

\end{document}